\documentclass[10pt]{article}
\usepackage{latexsym,amssymb,amsmath}
\usepackage{amssymb,amsthm,upref,amscd}
\usepackage{color}
\usepackage[T1]{fontenc}
\begin{document}
\baselineskip=14pt

\numberwithin{equation}{section}

\newtheorem{thm}{Theorem}[section]
\newtheorem{lem}[thm]{Lemma}
\newtheorem{cor}[thm]{Corollary}
\newtheorem{Prop}[thm]{Proposition}
\newtheorem{Def}[thm]{Definition}
\newtheorem{Rem}[thm]{Remark}
\newtheorem{Ex}[thm]{Example}

\newcommand{\A}{\mathbb{A}}
\newcommand{\B}{\mathbb{B}}
\newcommand{\C}{\mathbb{C}}
\newcommand{\D}{\mathbb{D}}
\newcommand{\E}{\mathbb{E}}
\newcommand{\F}{\mathbb{F}}
\newcommand{\G}{\mathbb{G}}
\newcommand{\I}{\mathbb{I}}
\newcommand{\J}{\mathbb{J}}
\newcommand{\K}{\mathbb{K}}
\newcommand{\M}{\mathbb{M}}
\newcommand{\N}{\mathbb{N}}
\newcommand{\Q}{\mathbb{Q}}
\newcommand{\R}{\mathbb{R}}
\newcommand{\T}{\mathbb{T}}
\newcommand{\U}{\mathbb{U}}
\newcommand{\V}{\mathbb{V}}
\newcommand{\W}{\mathbb{W}}
\newcommand{\X}{\mathbb{X}}
\newcommand{\Y}{\mathbb{Y}}
\newcommand{\Z}{\mathbb{Z}}
\newcommand\ca{\mathcal{A}}
\newcommand\cb{\mathcal{B}}
\newcommand\cc{\mathcal{C}}
\newcommand\cd{\mathcal{D}}
\newcommand\ce{\mathcal{E}}
\newcommand\cf{\mathcal{F}}
\newcommand\cg{\mathcal{G}}
\newcommand\ch{\mathcal{H}}
\newcommand\ci{\mathcal{I}}
\newcommand\cj{\mathcal{J}}
\newcommand\ck{\mathcal{K}}
\newcommand\cl{\mathcal{L}}
\newcommand\cm{\mathcal{M}}
\newcommand\cn{\mathcal{N}}
\newcommand\co{\mathcal{O}}
\newcommand\cp{\mathcal{P}}
\newcommand\cq{\mathcal{Q}}
\newcommand\rr{\mathcal{R}}
\newcommand\cs{\mathcal{S}}
\newcommand\ct{\mathcal{T}}
\newcommand\cu{\mathcal{U}}
\newcommand\cv{\mathcal{V}}
\newcommand\cw{\mathcal{W}}
\newcommand\cx{\mathcal{X}}
\newcommand\ocd{\overline{\cd}}

\def\c{\centerline}
\def\ov{\overline}
\def\emp {\emptyset}
\def\pa {\partial}
\def\bl{\setminus}
\def\op{\oplus}
\def\sbt{\subset}
\def\un{\underline}
\def\al {\alpha}
\def\bt {\beta}
\def\de {\delta}
\def\Ga {\Gamma}
\def\ga {\gamma}
\def\lm {\lambda}
\def\Lam {\Lambda}
\def\om {\omega}
\def\Om {\Omega}
\def\sa {\sigma}
\def\vr {\varepsilon}
\def\va {\varphi}

\title{\bf Uniqueness and nondegeneracy of solutions for a critical nonlocal equation\thanks{Partially supported by NSFC (11571317).}}

\author {Lele Du$^{1}$, Minbo Yang$^{1}$\thanks{M. Yang is the corresponding author: mbyang@zjnu.edu.cn}
\\
\\
{\small 1 Department of Mathematics, Zhejiang Normal University, Jinhua, Zhejiang, 321004, P. R. China}}

\date{}
\maketitle

\begin{abstract}
The aim of this paper is to classify the positive solutions of the nonlocal critical equation:
$$
-\Delta u=\left(I_{\mu}\ast u^{2^{\ast}_{\mu}}\right)u^{{2}^{\ast}_{\mu}-1},~~x\in\mathbb{R}^{N},
$$
where $0<\mu<N$, if $N=3\ \hbox{or} \ 4$ and $0<\mu\leq4$ if $N\geq5$,
$I_{\mu}$ is the Riesz potential defined by
$$
I_{\mu}(x)=\frac{\Gamma(\frac{\mu}{2})}{\Gamma(\frac{N-\mu}{2})\pi^{\frac{N}{2}}2^{N-\mu}|x|^{\mu}}
$$
with $\Gamma(s)=\int^{+\infty}_{0}x^{s-1}e^{-x}dx$, $s>0$ and $2^{\ast}_{\mu}=\frac{2N-\mu}{N-2}$ is the upper critical exponent in the sense of the Hardy-Littlewood-Sobolev inequality. We apply the moving plane method in integral forms to prove the symmetry and uniqueness of the positive solutions and prove the nondegeneracy of the unique solutions for the equation when $\mu$ close to $N$.
 \vspace{0.3cm}


\vspace{0.3cm}

\end{abstract}

\section{Introduction}

In this paper we are interested in classifying the positive solutions of the critical equation
\begin{equation}\label{cc}
-\Delta u=\left(I_{\mu}\ast u^{2^{\ast}_{\mu}}\right)u^{{2}^{\ast}_{\mu}-1},~~x\in\mathbb{R}^{N},
\end{equation}
where $2^{\ast}_{\mu}=\frac{2N-\mu}{N-2}$, $I_{\mu}$ is the Riesz potential defined by
$$
I_{\mu}(x)=\frac{\Gamma(\frac{\mu}{2})}{\Gamma(\frac{N-\mu}{2})\pi^{\frac{N}{2}}2^{N-\mu}|x|^{\mu}}
$$
with $\Gamma(s)=\int^{+\infty}_{0}x^{s-1}e^{-x}dx$, $s>0$ \Big(In some references the Riesz potential is defined by $
I_{\alpha}(x)=\frac{\Gamma(\frac{N-\alpha}{2})}{\Gamma(\frac{\alpha}{2})\pi^{\frac{N}{2}}2^{\alpha}|x|^{N-\alpha}}
$\Big). By rescaling, we know equation \eqref{cc} is equivalent to
\begin{equation}\label{UJM}
-\Delta u=\left(\frac{1}{|x|^{\mu}}\ast u^{{2}^{\ast}_{\mu}}\right)u^{{2}^{\ast}_{\mu}-1},~~x\in\mathbb{R}^{N},
\end{equation}
which is related to the topic of nonlinear Choquard equation
\begin{equation}\label{ce}
-\Delta u+V(x)u=\left(\frac{1}{|x|^{\mu}}\ast u^{p}\right)u^{p-1},~~x\in\mathbb{R}^{N}.
\end{equation}
The Choquard equation arises from various domains of mathematical physics, such as physics of many-particle systems, quantum mechanics, physics of laser beams and so on.
Mathematically, for equation \eqref{ce} with $\mu=1$, $p=2$ and $V$ is a positive constant, Lieb \cite{LE2} proved the existence and uniqueness of the ground
state by using rearrangements technique. Lions \cite{Ls} obtained the existence of a sequence of radially symmetric solutions by variational methods.  Ma and Zhao \cite{MZ} proved that the positive solutions of the stationary Choquard equation are uniquely determined up to translations by the methods in \cite{CLO2}. For suitable range of the exponent $p$, the authors \cite{MZ, MS1} also showed the regularity, positivity and radial symmetry of the ground states and
derived decay property at infinity as well.

 To understand the critical growth for equation \eqref{cc}, we need to recall the well-known Hardy-Littlewood-Sobolev (HLS for short) inequality, see \cite{LE1, LL}.
\begin{Prop}\label{HLS}
 Let $t,r>1$ and $0<\mu<N$ satisfying $1/t+1/r+\mu/N=2$, $f\in L^{t}(\mathbb{R}^N)$ and $h\in L^{r}(\mathbb{R}^N)$. There exists a sharp constant $C(N,\mu,t,r)$, independent of $f,h$, such that
\begin{equation}\label{HLS1}
\int_{\mathbb{R}^{N}}\int_{\mathbb{R}^{N}}\frac{f(x)h(y)}{|x-y|^{\mu}}dxdy\leq C(N,\mu,t,r) |f|_{t}|h|_{r}.
\end{equation}
If $t=r=2N/(2N-\mu)$, then
$$
C(N,\mu,t,r)=C(N,\mu)=\pi^{\frac{\mu}{2}}\frac{\Gamma(\frac{N-\mu}{2})}{\Gamma(N-\frac{\mu}{2})}\left\{\frac{\Gamma(N)}{\Gamma(\frac{N}{2})}\right\}^{\frac{N-\mu}{N}}
$$
and there is equality in \eqref{HLS1} if and only if $f\equiv(const.)h$ and
$$
h(x)=A(\gamma^{2}+|x-a|^{2})^{-(2N-\mu)/2}
$$
for some $A\in \mathbb{C}$, $0\neq\gamma\in\mathbb{R}$ and $a\in \mathbb{R}^{N}$.
\end{Prop}

In fact, by the HLS inequality, we know
$$
\int_{\mathbb{R}^{N}}\int_{\mathbb{R}^{N}}\frac{u(x)^{p}u(y)^{p}}{|x-y|^{\mu}}dxdy
$$
is well-defined if for $t>1$, $u^{p}\in L^{t}(\mathbb{R}^{N})$ satisfying
$$
\frac{2}{t}+\frac{\mu}{N}=2.
$$
Therefore, for $u\in H^{1}(\mathbb{R}^{N})$, the Sobolev embedding implies
$$
\frac{2N-\mu}{N}\leq p\leq\frac{2N-\mu}{N-2}.
$$
Thus we may call $\frac{2N-\mu}{N}$ the lower critical exponent and $\frac{2N-\mu}{N-2}$ the upper critical exponent due to the HLS inequality. Lieb \cite{LE1} classified all the maximizers of the HLS functional under constraints and obtained the best constant, then he posed the classification of the solutions of
\begin{equation}\label{Int}
u(x)=\int_{\mathbb{R}^{N}}\frac{u(y)^{\frac{N+\alpha}{N-\alpha}}}{|x-y|^{N-\alpha}}dy,~~x\in\mathbb{R}^{N},
\end{equation}
as an open problem. Later, Chen, Li and Ou \cite{CLO1} developed the method of moving planes in integral forms to prove that any critical points of the functional was radially symmetric and assumed the unique form and gave a positive answer to Lieb's open problem.
In fact, equation \eqref{Int} is closely related to the well-known fractional equation
\begin{equation}\label{Frac}
(-\Delta)^{\frac{\alpha}{2}}u=u^{\frac{N+\alpha}{N-\alpha}},~~x\in\mathbb{R}^{N}.
\end{equation}
While Lu and Zhu \cite{LZ} studied the symmetry and
regularity of extremals of the following integral equation:
\begin{equation}\label{Int2}
u(x)=\int_{\mathbb{R}^{N}}\frac{u(y)^{\frac{(N+\alpha-2s)}{N-\alpha}}}{|y|^s|x-y|^{N-\alpha}}dy,~~x\in\mathbb{R}^{N},
\end{equation}
which is related to the fractional singular case
\begin{equation}\label{Frac2}
(-\Delta)^{\frac{\alpha}{2}}u=\frac{u^{\frac{N+\alpha-2s}{N-\alpha}}}{|x|^s},~~x\in\mathbb{R}^{N}.
\end{equation}
If $\alpha=2$ and $0<s<2$, equation \eqref{Frac2} is closely related to the Euler-Lagrange equation
of the extremal functions of the classical Hardy-Sobolev inequality which says that there exists a constant $C>0$ such that
\begin{equation}\label{HS}
\big(\int_{\mathbb{R}^{N}}\frac{u^{\frac{2(N-s)}{N-\alpha}}}{|x|^s}dx\big)^{\frac{N-\alpha}{N-s}}\leq C \int_{\mathbb{R}^{N}}|\nabla u|^{2}dx.
\end{equation}
When $N\geq3$, $\alpha=2$ and $s=0$, both equation \eqref{Frac} and \eqref{Frac2} go back to
\begin{equation}\label{lcritical}
-\Delta u=u^{\frac{N+2}{N-2}},~~x\in\mathbb{R}^{N}.
\end{equation}
Equation \eqref{lcritical} is related to the Euler-Lagrange equation
of the extremal functions of the Sobolev inequality  and is a special case of the Lane-Emden equation
\begin{equation}\label{LE}
-\Delta u=u^{p},~~x\in\mathbb{R}^{N}.
\end{equation}
It is well known that, for $0<p<\frac{N+2}{N-2}$, Gidas and Spruck \cite{GS} proved that \eqref{LE} has no positive solutions. This
result is optimal in the sense that for any $p\geq\frac{N+2}{N-2}$, there are infinitely many positive solutions
to \eqref{LE}. Gidas, Ni and Nirenberg \cite{GNN}, Caffarelli, Gidas and Spruck \cite{CGS} proved the symmetry and uniqueness of the positive solutions respectively. Chen and Li \cite{CL}, Li \cite{LC} simplified the results above as an application of the moving plane method. Li \cite{Li} used moving sphere method. The classification of the solutions of equation \eqref{lcritical} plays an important role in the Yamabe problem, the prescribed scalar curvature problem on Riemannian manifolds and the priori estimates in nonlinear equations. Aubin \cite{A}, Talenti \cite{Ta} proved that the best Sobolev constant $S$
 can be achieved by a two-parameter solutions of the form
\begin{equation}\label{U0}
 U_0(x):=[N(N-2)]^{\frac{N-2}{4}}\Big(\frac{t}{t^2+|x-\xi|^{2}}\Big)^{\frac{N-2}{2}}.
 \end{equation}
While Lieb \cite{LE1} proved that the best constant in \eqref{HS} is achieved and the extremal
function is identified by
$$\frac{1}{(1+|x|^{2-t})^{\frac{N-2}{2-t}}}.$$
Furthermore, note that equation \eqref{lcritical}
has a $(N+1)$-dimensional manifold of solutions given by
$$
\mathcal{Z}=\left\{z_{t,\xi}=[N(N-2)]^{\frac{N-2}{4}}\Big(\frac{t}{t^2+|x-\xi|^{2}}\Big)^{\frac{N-2}{2}},
\xi\in\mathbb{R}^{N}, t\in\mathbb{R}^{+}\right\}.
$$
For every $Z\in \mathcal{Z}$, it is said to be nondegenerate in the sense that the linearized equation
\begin{equation}\label{Linearized}
-\Delta v=Z^{\frac{4}{N-2}}v
\end{equation}
in $D^{1,2}(\mathbb{R}^N)$ only admits solutions of the form
$$
\eta=aD_{t}Z+\mathbf{b}\cdot\nabla Z,
$$
where $a\in\mathbb{R},\mathbf{b}\in\mathbb{R}^{N}$.

Regarding to equation \eqref{cc}, we find that it is a special case of the nonlocal equation
\begin{equation}\label{cc2}
-\Delta u=(I_{\mu}\ast u^{p})u^{p-1},~~x\in\mathbb{R}^{N},
\end{equation}
where $\frac{2N-\mu}{N}\leq p\leq\frac{2N-\mu}{N-2}$.  By the convergence property of the Riesz potential that $I_{\mu}\to \delta_x$ as $\mu\to N$, we find that equation \eqref{cc2} goes back to equation \eqref{LE} as $\mu\to N$ \cite{Lan}. Thus equation \eqref{LE} can also be treated as the limit equation\eqref{cc2}. In this sense, we may consider equation \eqref{cc} as a generalization of equation \eqref{lcritical} from a nonlocal point of view. Moreover, from the HLS inequality and Sobolev inequality, we know equation \eqref{cc} is related to the Euler-Lagrange equation
of the extremal functions of the limit embedding.

In the following we define the best constant $S^{\ast}_{H,L}$ by
\begin{equation}\label{YHNT}
S^{\ast}_{H,L}=\inf_{u\in D^{1,2}(\mathbb{R}^{N})\backslash\{0\}}\frac{\displaystyle\int_{\mathbb{R}^{N}}|\nabla u|^{2}dx}{\left(\displaystyle\int_{\mathbb{R}^{N}}(I_{\mu}\ast |u|^{2^{\ast}_{\mu}})|u|^{2^{\ast}_{\mu}}dy\right)^{\frac{N-2}{2N-\mu}}}.
\end{equation}
After rescaling, we know that the extremal functions for the best constant $S^{\ast}_{H,L}$ also satisfy
$$
-\Delta u=\left(I_{\mu}\ast u^{2^{\ast}_{\mu}}\right)u^{{2}^{\ast}_{\mu}-1},~~x\in\mathbb{R}^{N}.
$$
Notice that, for all $u\in D^{1,2}(\mathbb{R}^{N})$,
\begin{equation}\nonumber
\begin{aligned}
\left(\int_{\mathbb{R}^{N}}(I_{\mu}\ast |u|^{2^{\ast}_{\mu}})|u|^{2^{\ast}_{\mu}}dy\right)^{\frac{N-2}{2N-\mu}}
&=\left(\frac{\Gamma(\frac{\mu}{2})}{\Gamma(\frac{N-\mu}{2})\pi^{\frac{N}{2}}2^{N-\mu}}\right)^{\frac{N-2}{2N-\mu}}
\left(\int_{\mathbb{R}^{N}}\int_{\mathbb{R}^{N}}\frac{|u(x)|^{2^{\ast}_{\mu}}|u(y)|^{2^{\ast}_{\mu}}}{|x-y|^{\mu}}dxdy\right)^{\frac{N-2}{2N-\mu}}\\
&\leq\left(\frac{\Gamma(\frac{\mu}{2})}{\Gamma(\frac{N-\mu}{2})\pi^{\frac{N}{2}}2^{N-\mu}}\pi^{\frac{\mu}{2}}\frac{\Gamma(\frac{N-\mu}{2})}{\Gamma(N-\frac{\mu}{2})}\left\{\frac{\Gamma(N)}{\Gamma(\frac{N}{2})}\right\}^{\frac{N-\mu}{N}}\right)^{\frac{N-2}{2N-\mu}}||u||_{L^{2^{\ast}}}^{2}\\
&=\left(\left(\frac{1}{2\sqrt{\pi}}\right)^{N-\mu}\frac{\Gamma(\frac{\mu}{2})}{\Gamma(N-\frac{\mu}{2})}\left\{\frac{\Gamma(N)}{\Gamma(\frac{N}{2})}\right\}^{\frac{N-\mu}{N}}\right)^{\frac{N-2}{2N-\mu}}||u||_{L^{2^{\ast}}}^{2}.\\
\end{aligned}
\end{equation}
Let $C^{\ast}(N,\mu):=\left(\frac{1}{2\sqrt{\pi}}\right)^{N-\mu}\frac{\Gamma(\frac{\mu}{2})}{\Gamma(N-\frac{\mu}{2})}\left\{\frac{\Gamma(N)}{\Gamma(\frac{N}{2})}\right\}^{\frac{N-\mu}{N}}$,   then $C^{\ast}(N,\mu)\to1$ as $\mu\to N$, Moreover we have
\begin{equation}\label{ABCD}
\left(\int_{\mathbb{R}^{N}}\left(I_{\mu}\ast |u|^{2^{\ast}_{\mu}}\right)|u|^{2^{\ast}_{\mu}}dy\right)^{\frac{N-2}{2N-\mu}}
\leq C^{\ast}(N,\mu)^{\frac{N-2}{2N-\mu}}||u||_{L^{2^{\ast}}}^{2}.\\
\end{equation}

Thus we may conclude the following lemma:
\begin{lem}\label{lem1}
The constant $S^{\ast}_{H,L}$ defined in \eqref{YHNT} satisfies
\begin{equation}\label{WSX}
S^{\ast}_{H,L}=\frac{S}{C^{\ast}(N,\mu)^{\frac{N-2}{2N-\mu}}}.
\end{equation}
Where $S$ is the Sobolev constant and
$$
C^{\ast}(N,\mu)=\left(\frac{1}{2\sqrt{\pi}}\right)^{N-\mu}\frac{\Gamma(\frac{\mu}{2})}{\Gamma(N-\frac{\mu}{2})}\left\{\frac{\Gamma(N)}{\Gamma(\frac{N}{2})}\right\}^{\frac{N-\mu}{N}}.
$$
What's more,
$$
{U}_\mu(x)=S^{\frac{(N-\mu)(2-N)}{4(N-\mu+2)}}C^{\ast}(N,\mu)^{\frac{2-N}{2(N-\mu+2)}}U_0
$$
is the unique family of minimizers for $S^{\ast}_{H,L}$ that satisfies equation \eqref{cc}.
\end{lem}
\begin{proof}
From \eqref{ABCD}, we know that
\begin{equation}\nonumber
\begin{aligned}
S^{\ast}_{H,L}&=\inf_{u\in D^{1,2}(\mathbb{R}^{N})\backslash\{0\}}\frac{\displaystyle\int_{\mathbb{R}^{N}}|\nabla u|^{2}dx}{\left(\displaystyle\int_{\mathbb{R}^{N}}\left(I_{\mu}\ast |u|^{2^{\ast}_{\mu}}\right)|u|^{2^{\ast}_{\mu}}dy\right)^{\frac{N-2}{2N-\mu}}}\\
&\geq\frac{1}{C^{\ast}(N,\mu)^{\frac{N-2}{2N-\mu}}}\inf_{u\in D^{1,2}(\mathbb{R}^{N})\backslash\{0\}}\frac{\displaystyle\int_{\mathbb{R}^{N}}|\nabla u|^{2}dx}{||u||_{L^{2^{\ast}}}^{2}}\\
&=\frac{S}{C^{\ast}(N,\mu)^{\frac{N-2}{2N-\mu}}}.
\end{aligned}
\end{equation}
Meanwhile, we obtain from the definition of $S^{\ast}_{H,L}$ that
\begin{equation}\nonumber
\begin{aligned}
S^{\ast}_{H,L}&=\inf_{u\in D^{1,2}(\mathbb{R}^{N})\backslash\{0\}}\frac{\displaystyle\int_{\mathbb{R}^{N}}|\nabla u|^{2}dx}{\left(\displaystyle\int_{\mathbb{R}^{N}}\left(I_{\mu}\ast |u|^{2^{\ast}_{\mu}}\right)|u|^{2^{\ast}_{\mu}}dy\right)^{\frac{N-2}{2N-\mu}}}\\
&\leq\frac{\displaystyle\int_{\mathbb{R}^{N}}|\nabla {U_\mu}|^{2} dx}{\left(\displaystyle\int_{\mathbb{R}^{N}}\left(I_{\mu}\ast {U}_\mu^{2^{\ast}_{\mu}}\right){U}_\mu^{2^{\ast}_{\mu}}dy\right)^{\frac{N-2}{2N-\mu}}}\\
&=\frac{1}{C^{\ast}(N,\mu)^{\frac{N-2}{2N-\mu}}}\frac{\displaystyle\int_{\mathbb{R}^{N}}|\nabla {U_\mu}|^{2} dx}{||{U}_\mu||_{L^{2^{\ast}}}^{2}}\\
&=\frac{S}{C^{\ast}(N,\mu)^{\frac{N-2}{2N-\mu}}}.
\end{aligned}
\end{equation}
Hence $S^{\ast}_{H,L}$ satisfies \eqref{WSX}. Moreover, we know ${U}_\mu$ is the unique minimizer for the best constant $S^{\ast}_{H,L}$ and satisfies equation \eqref{cc}.
\end{proof}
However, the classification of the solutions of equation \eqref{cc} still remains open, one may wonder that if $U_{\mu}(x)$ is the unique positive solution of equation \eqref{cc} and is nondegenerate in the sense that the linearized equation \eqref{Linearized} only admits solutions $D^{1,2}(\mathbb{R}^{N})$ of the form
$$
\eta=aD_{t}{U}_\mu+\mathbf{b}\cdot\nabla {U}_\mu,
$$
where $a\in\mathbb{R},\mathbf{b}\in\mathbb{R}^{N}$.
Inspired by the results about the local case, in this paper we are going to classify the positive solutions of the nonlinear Choquard equation \eqref{cc} with upper critical exponent and prove the nondegeneracy of this unique family of solutions.

 We are interested in the existence of weak solution $u(x)\in D^{1,2}(\mathbb{R}^{N})$ satisfying
\begin{equation}\label{YHN}
\int_{\mathbb{R}^{N}}\nabla u\nabla \nu dx=\int_{\mathbb{R}^{N}}\left(I_{\mu}\ast |u|^{2^{\ast}_{\mu}}\right)u^{{2}^{\ast}_{\mu}-1}\nu dx,~~\forall \nu\in C^{\infty}_{0}(\mathbb{R}^{N}).
\end{equation}
By Theorem 4.5 of \cite{CLO1} we know equation \eqref{cc} is equivalent to the integral equation
\begin{equation}\label{OPQ}
u(x)=\int_{\mathbb{R}^{N}}\frac{1}{|x-y|^{N-2}}\left(\int_{\mathbb{R}^{N}}\frac{u(z)^{2^{\ast}_{\mu}}}{|y-z|^{\mu}}dz\right)u(y)^{{2}^{\ast}_{\mu}-1}dy.
\end{equation}
Thus we can rewrite equation \eqref{OPQ} into an equivalent system
\begin{equation}\label{Sys}
\begin{cases}
u(x)=\displaystyle\int_{\mathbb{R}^{N}}\frac{v(y)u(y)^{{2}^{\ast}_{\mu}-1}}{|x-y|^{N-2}}dy,\vspace{3mm}\\
\displaystyle v(x)=\int_{\mathbb{R}^{N}}\frac{u(y)^{2^{\ast}_{\mu}}}{|x-y|^{\mu}}dy.
\end{cases}
\end{equation}

The main results of this paper are stated as
\begin{thm}\label{TH}
Assume that $N\geq3$, $0<\mu<N$, if $N=3\ \hbox{or} \ 4$ while $0<\mu\leq4$ if $N\geq5$, and $2^{\ast}_{\mu}=\frac{2N-\mu}{N-2}$, let $(u,v)\in L^{2^{\ast}}_{loc}(\mathbb{R}^{N})\times L^{\frac{2N}{\mu}}_{loc}(\mathbb{R}^{N})$ be a pair of positive solutions of system \eqref{Sys}. Then $u$ and $v$ must be radially symmetric, moreover, $u$ assumes the form
\begin{equation}\label{class}
c\left(\frac{t}{t^{2}+|x-x_{0}|^{2}}\right)^\frac{{N-2}}{2}
\end{equation}
about some point $x_{0}$ with positive constants $c$ and $t$.
\end{thm}

When the paper was finished, we notice that a special case of Theorem \ref{TH} was proved in \cite{Liu} for Hartree equation. For any $u\in L^{2^{\ast}}(\mathbb{R}^{N})$, by the HLS inequality we know that $v\in L^{\frac{2N}{\mu}}(\mathbb{R}^{N})$. Based on the uniqueness theorem for equation \eqref{OPQ}, we have
\begin{Rem}\label{TH2}
Assume that $N\geq3$, $0<\mu<N$, if $N=3\ \hbox{or} \ 4$ while $0<\mu\leq4$ if $N\geq5$, and $2^{\ast}_{\mu}=\frac{2N-\mu}{N-2}$. Let $u\in D^{1,2}(\mathbb{R}^{N})$ be a positive weak solution of equation \eqref{cc}, then there must exist $t, \xi$ such that $u=U_\mu$. Moreover, for fixed $t, \xi$ we have
$$
U_\mu\to U_0, \ \ \hbox{as}\ \ \mu\to N.
$$
\end{Rem}
The next result is about the nondegeneracy of $U_\mu$ as $\mu$ close to $N$. We have
\begin{thm}\label{TH3}
Assume that $0<\mu<N$,  $N=3\ \hbox{or} \ 4$  and $2^{\ast}_{\mu}=\frac{2N-\mu}{N-2}$. Let $\mu$ be close to $N$ then the linearized equation at $U_\mu$ \begin{equation}\label{Linearized E}
-\Delta\psi-2^{\ast}_{\mu}(I_{\mu}\ast(U_\mu^{2^{\ast}_{\mu}-1}\psi))U_\mu^{{2}^{\ast}_{\mu}-1}-(2^{\ast}_{\mu}-1)(I_{\mu}\ast U_\mu^{2^{\ast}_{\mu}})U_\mu^{{2}^{\ast}_{\mu}-2}\psi=0
\end{equation}
only admits solutions in $D^{1,2}(\mathbb{R}^{N})$ of the form
$$
\psi=aD_{t}{U}_\mu+\mathbf{b}\cdot\nabla {U}_\mu,
$$
where $a\in\mathbb{R},\mathbf{b}\in\mathbb{R}^{N}$.
\end{thm}
Before proving the main results, we notice that $u(x)$ is only assumed to be locally integrable and the nonlocal convolution of $u(x)$ under Riesz potential brings a lot of difficulty in estimating the nonlinearity at some point. Based on the above reasons, we will establish an integral inequality in section 2 to ensure that the planes can be moved, then the moving plane method in integral forms are applied to prove the symmetry of the positive solutions. In section 3, We will also investigate the unique form of solutions.  In section 4 we prove the nondegeneracy of the unique solutions when $\mu$ close to $N$.

\section{Symmetry}
The moving plane method was invented by
Alexandrov in the 1950s and has been further developed by many people. Among the existing results, Chen
et al. \cite{CLO1} applied the moving plane method to integral equations and obtained the symmetry,
monotonicity and nonexistence properties of the solutions.
 By virtue of the HLS inequality or
the weighted HLS inequality, moving plane method in integral form can be used to explore various specific features
of the integral equation itself without the help of the maximum principle of
differential equation. We may refer the readers to \cite{CLL, CLZ} for recent progress of the moving plane methods.
Furthermore, the qualitative analysis of the solutions for nonlocal equations has been studied in \cite{CLO3, CL3, LY, Le, LLM}.

Since there is no decay assumption of $u(x)$ and $v(x)$ at infinity, we are not able to apply the method of moving planes to $u(x)$ and $v(x)$ directly. An effective way to overcome the difficulty is to apply the Kelvin-type transform of $u(x),v(x)$ for any $x\neq0$ as
$$
s(x)=\frac{1}{|x|^{N-2}}u\left(\frac{x}{|x|^{2}}\right),~~t(x)=\frac{1}{|x|^{\mu}}v\left(\frac{x}{|x|^{2}}\right),
$$
we know that the origin is the possible singularity of $s(x)$ and $t(x)$.  By means of the Kelvin-type transform, system \eqref{Sys} is equivalent to
\begin{equation}\label{Ksys}
\begin{cases}
s(x)=\displaystyle\int_{\mathbb{R}^{N}}\frac{t(y)s(y)^{{2}^{\ast}_{\mu}-1}}{|x-y|^{N-2}}dy,\vspace{3mm}\\
\displaystyle t(x)=\int_{\mathbb{R}^{N}}\frac{s(y)^{2^{\ast}_{\mu}}}{|x-y|^{\mu}}dy.
\end{cases}
\end{equation}
One may find that the system is invariable under the Kelvin-type transform. To apply the method of moving planes, we introduce the following symbols
$$
x^{\lambda}=(2\lambda-x_{1},...,x_{N});~~s(x^{\lambda})=s_{\lambda}(x);~~t(x^{\lambda})=t_{\lambda}(x)
$$
and the planes
$$
\Sigma_{\lambda}=\{x\in\mathbb{R}^{N}:x_{1}>\lambda\};~~T_{\lambda}=\{x\in\mathbb{R}^{N}:x_{1}=\lambda\};
$$
$$
\Sigma^{s}_{\lambda}=\{x\in\Sigma_{\lambda}-\{{0^{\lambda}}\}|s(x)>s(x^{\lambda})\};~~\Sigma^{t}_{\lambda}=\{x\in\Sigma_{\lambda}-\{{0^{\lambda}}\}|t(x)>t(x^{\lambda})\}.
$$

To apply the method of moving planes, we should first study the integrability of $s(x)$ and $t(x)$.
\begin{lem}\label{Inter1}
For any $\varepsilon>0$, $s\in L^{2^{\ast}}(\mathbb{R}^{N}-B_{\varepsilon}(0))$ and $t\in L^{\frac{2N}{\mu}}(\mathbb{R}^{N}-B_{\varepsilon}(0))$.
\end{lem}
\begin{proof} Choose $y=\frac{x}{|x|^{2}}$, then we have $|y|\cdot|x|=1$ and $dy=|x|^{-2N}dx$. For $x\in\mathbb{R}^{N}-B_{\varepsilon}(0)$, there exists a bounded domain $\Omega\subset B_{\frac{1}{\varepsilon}}(0)$, such that $y\in\Omega$. It is easy to prove that
\begin{equation}\nonumber
\begin{aligned}
\int_{\mathbb{R}^{N}-B_{\varepsilon}(0)}s(x)^{2^{\ast}}dx&=\int_{\mathbb{R}^{N}-B_{\varepsilon}(0)}(|x|^{2-N})^{2^{\ast}}u\left(\frac{x}{|x|^{2}}\right)^{2^{\ast}}dx\\
&=\int_{\Omega}u(y)^{2^{\ast}}dy.
\end{aligned}
\end{equation}
Since $u\in L^{2^{\ast}}_{loc}(\mathbb{R}^{N})$, we get the conclusion that $s\in L^{2^{\ast}}(\mathbb{R}^{N}-B_{\varepsilon}(0))$. On the other hand, we can use the HLS inequality to verify that
\begin{equation}\nonumber
\begin{aligned}
\int_{\mathbb{R}^{N}-B_{\varepsilon}(0)}t(x)^{\frac{2N}{\mu}}dx&=\int_{\mathbb{R}^{N}-B_{\varepsilon}(0)}(|x|^{-\mu})^{\frac{2N}{\mu}}v\left(\frac{x}{|x|^{2}}\right)^{\frac{2N}{\mu}}dx\\
&=\int_{\Omega}v(y)^{\frac{2N}{\mu}}dy.
\end{aligned}
\end{equation}
This implies the desire conclusion since $v\in L^{\frac{2N}{\mu}}_{loc}(\mathbb{R}^{N})$.
\end{proof}

First we consider the case $0<\mu<\min\{4,N\}$ and have the following conclusion.
\begin{lem}\label{Inter2}
Suppose that $0<\mu<\min\{4,N\}$, then for any $\lambda>0$ there exists a constant $C>0$ such that
$$
||s-s_{\lambda}||_{L^{2^{\ast}}({\Sigma^{s}_{\lambda}})}\leq C\left[\|t\|_{L^{\frac{2N}{\mu}}(\Sigma^{s}_{\lambda})}\|s\|^{2^{\ast}_{\mu}-2}_{L^{2^{\ast}}(\Sigma^{s}_{\lambda})}
+||s||^{{2}^{\ast}_{\mu}-1}_{L^{2^{\ast}}(\Sigma^{t}_{\lambda})}
||s||^{2^{\ast}_{\mu}-1}_{L^{2^{\ast}}(\Sigma^{s}_{\lambda})}\right]||s-s_{\lambda}||_{L^{2^{\ast}}(\Sigma^{s}_{\lambda})}.
$$
\end{lem}
\begin{proof}
For any $\lambda>0$,
\begin{equation}\nonumber
\begin{aligned}
s(x)&=\int_{\mathbb{R}^{N}}\frac{t(y)s(y)^{{2}^{\ast}_{\mu}-1}}{|x-y|^{N-2}}dy\\
&=\int_{\Sigma_{\lambda}}\frac{t(y)s(y)^{{2}^{\ast}_{\mu}-1}}{|x-y|^{N-2}}dy
+\int_{\mathbb{R}^{N}-\Sigma_{\lambda}}\frac{t(y)s(y)^{{2}^{\ast}_{\mu}-1}}{|x-y|^{N-2}}dy\\
&=\int_{\Sigma_{\lambda}}\frac{t(y)s(y)^{{2}^{\ast}_{\mu}-1}}{|x-y|^{N-2}}dy
+\int_{\Sigma_{\lambda}}\frac{t(y^{\lambda})s(y^{\lambda})^{{2}^{\ast}_{\mu}-1}}{|x-y^{\lambda}|^{N-2}}dy.
\end{aligned}
\end{equation}
Similarly,
$$
s_{\lambda}(x)=\int_{\Sigma_{\lambda}}\frac{t(y)s(y)^{{2}^{\ast}_{\mu}-1}}{|x^{\lambda}-y|^{N-2}}dy
+\int_{\Sigma_{\lambda}}\frac{t(y^{\lambda})s(y^{\lambda})^{{2}^{\ast}_{\mu}-1}}{|x^{\lambda}-y^{\lambda}|^{N-2}}dy.
$$
Since $|x^{\lambda}-y^{\lambda}|=|x-y|$ and $|x-y^{\lambda}|=|x^{\lambda}-y|$, we know
\begin{equation}\label{Minus}
s(x)-s_{\lambda}(x)=\int_{\Sigma_{\lambda}}\left[\frac{1}{|x-y|^{N-2}}-\frac{1}{|x^{\lambda}-y|^{N-2}}\right]\left[t(y)s(y)^{{2}^{\ast}_{\mu}-1}-t(y^{\lambda})s(y^{\lambda})^{{2}^{\ast}_{\mu}-1}\right]dy.
\end{equation}
If we divide the domain of integration into four disjoint parts as
$$
\Sigma_{\lambda}=\{\Sigma^{s}_{\lambda}\cap\Sigma^{t}_{\lambda}\}\cup\{\Sigma^{s}_{\lambda}-\Sigma^{t}_{\lambda}\}\cup\{\Sigma^{t}_{\lambda}-\Sigma^{s}_{\lambda}\}\cup
\{\Sigma_{\lambda}-\Sigma^{s}_{\lambda}-\Sigma^{t}_{\lambda}\},
$$
we have the following four cases.

1. If $y\in\Sigma^{s}_{\lambda}\cap\Sigma^{t}_{\lambda}$, we apply the Mean Value Theorem to obtain
\begin{equation}\nonumber
\begin{aligned}
ts^{{2}^{\ast}_{\mu}-1}-t_{\lambda}s_{\lambda}^{{2}^{\ast}_{\mu}-1}&=t[s^{{2}^{\ast}_{\mu}-1}-s_{\lambda}^{{2}^{\ast}_{\mu}-1}]+(t-t_{\lambda})s_{\lambda}^{{2}^{\ast}_{\mu}-1}\\
&\leq({2}^{\ast}_{\mu}-1)ts^{{2}^{\ast}_{\mu}-2}(s-s_{\lambda})+(t-t_{\lambda})s^{{2}^{\ast}_{\mu}-1}.
\end{aligned}
\end{equation}

2. If $y\in\Sigma^{s}_{\lambda}-\Sigma^{t}_{\lambda}$, we apply the Mean Value Theorem again to get
\begin{equation}\nonumber
\begin{aligned}
ts^{{2}^{\ast}_{\mu}-1}-t_{\lambda}s_{\lambda}^{{2}^{\ast}_{\mu}-1}&\leq ts^{{2}^{\ast}_{\mu}-1}-ts_{\lambda}^{{2}^{\ast}_{\mu}-1}\\
&=({2}^{\ast}_{\mu}-1)ts^{{2}^{\ast}_{\mu}-2}(s-s_{\lambda}).
\end{aligned}
\end{equation}

3. If $y\in\Sigma^{t}_{\lambda}-\Sigma^{s}_{\lambda}$, then
\begin{equation}\nonumber
\begin{aligned}
ts^{{2}^{\ast}_{\mu}-1}-t_{\lambda}s_{\lambda}^{{2}^{\ast}_{\mu}-1}&\leq ts^{{2}^{\ast}_{\mu}-1}-t_{\lambda}s^{{2}^{\ast}_{\mu}-1}\\
&=(t-t_{\lambda})s^{{2}^{\ast}_{\mu}-1}.
\end{aligned}
\end{equation}

4. If $y\in\Sigma_{\lambda}-\Sigma^{s}_{\lambda}-\Sigma^{t}_{\lambda}$, then
\begin{equation}\nonumber
\begin{aligned}
ts^{{2}^{\ast}_{\mu}-1}-t_{\lambda}s_{\lambda}^{{2}^{\ast}_{\mu}-1}&\leq0.
\end{aligned}
\end{equation}
Then equality \eqref{Minus} implies that for $x\in\Sigma_{\lambda}-\{{0^{\lambda}}\}$,
\begin{equation}\nonumber
\begin{aligned}
s(x)-s_{\lambda}(x)&\leq\int_{\Sigma_{\lambda}}\left[\frac{1}{|x-y|^{N-2}}-\frac{1}{|x^{\lambda}-y|^{N-2}}\right]
\left[({2}^{\ast}_{\mu}-1)ts^{{2}^{\ast}_{\mu}-2}(s-s_{\lambda})\chi_{\{\Sigma^{s}_{\lambda}\cap\Sigma^{t}_{\lambda}\}}
+s^{{2}^{\ast}_{\mu}-1}(t-t_{\lambda})\chi_{\{\Sigma^{s}_{\lambda}\cap\Sigma^{t}_{\lambda}\}}\right]dy\\
&\hspace{4mm}+\int_{\Sigma_{\lambda}}\left[\frac{1}{|x-y|^{N-2}}-\frac{1}{|x^{\lambda}-y|^{N-2}}\right]\left[({2}^{\ast}_{\mu}-1)ts^{{2}^{\ast}_{\mu}-2}(s-s_{\lambda})\chi_{\{\Sigma^{s}_{\lambda}-\Sigma^{t}_{\lambda}\}}\right]dy\\
&\hspace{6mm}+\int_{\Sigma_{\lambda}}\left[\frac{1}{|x-y|^{N-2}}-\frac{1}{|x^{\lambda}-y|^{N-2}}\right]s^{{2}^{\ast}_{\mu}-1}(t-t_{\lambda})\chi_{\{\Sigma^{t}_{\lambda}-\Sigma^{s}_{\lambda}\}}dy\\
&=\int_{\Sigma_{\lambda}}\left[\frac{1}{|x-y|^{N-2}}-\frac{1}{|x^{\lambda}-y|^{N-2}}\right]
\left[({2}^{\ast}_{\mu}-1)ts^{{2}^{\ast}_{\mu}-2}(s-s_{\lambda})\chi_{\{\Sigma^{s}_{\lambda}\}}+s^{{2}^{\ast}_{\mu}-1}(t-t_{\lambda})\chi_{\{\Sigma^{t}_{\lambda}\}}\right]dy\\
&\leq\int_{\Sigma_{\lambda}}\frac{1}{|x-y|^{N-2}}
\left[({2}^{\ast}_{\mu}-1)ts^{{2}^{\ast}_{\mu}-2}(s-s_{\lambda})\chi_{\{\Sigma^{s}_{\lambda}\}}+s^{{2}^{\ast}_{\mu}-1}(t-t_{\lambda})\chi_{\{\Sigma^{t}_{\lambda}\}}\right]dy\\
&=({2}^{\ast}_{\mu}-1)\int_{\Sigma^{s}_{\lambda}}\frac{ts^{{2}^{\ast}_{\mu}-2}(s-s_{\lambda})}{|x-y|^{N-2}}
dy+\int_{\Sigma^{t}_{\lambda}}\frac{s^{{2}^{\ast}_{\mu}-1}(t-t_{\lambda})}{|x-y|^{N-2}}dy.
\end{aligned}
\end{equation}

Involving $t(x)$, we also have
\begin{equation}\nonumber
\begin{aligned}
t(x)&=\int_{\mathbb{R}^{N}}\frac{1}{|x-y|^{\mu}}s(y)^{2^{\ast}_{\mu}}dy\\
&=\int_{\Sigma_{\lambda}}\frac{1}{|x-y|^{\mu}}s(y)^{2^{\ast}_{\mu}}dy+\int_{\mathbb{R}^{N}-\Sigma_{\lambda}}\frac{1}{|x-y|^{\mu}}s(y)^{2^{\ast}_{\mu}}dy\\
&=\int_{\Sigma_{\lambda}}\frac{1}{|x-y|^{\mu}}s(y)^{2^{\ast}_{\mu}}dy+\int_{\Sigma_{\lambda}}\frac{1}{|x-y^{\lambda}|^{\mu}}s(y^{\lambda})^{2^{\ast}_{\mu}}dy.
\end{aligned}
\end{equation}
Thus we have
\begin{equation}\label{IEquality}
t(x)-t(x^{\lambda})=\int_{\Sigma_{\lambda}}\left(\frac{1}{|x-y|^{\mu}}-\frac{1}{|x^{\lambda}-y|^{\mu}}\right)\left[s(y)^{2^{\ast}_{\mu}}-s(y^{\lambda})^{2^{\ast}_{\mu}}\right]dy.
\end{equation}
By the Mean Value Theorem, if $x\in\Sigma_{\lambda}-\{{0^{\lambda}}\}$ then
\begin{equation}\label{Inequality2}
\begin{aligned}
t(x)-t(x^{\lambda})&\leq\int_{\Sigma^{s}_{\lambda}}\frac{1}{|x-y|^{\mu}}[s(y)^{2^{\ast}_{\mu}}-s(y^{\lambda})^{2^{\ast}_{\mu}}]dy\\
&\leq2^{\ast}_{\mu}\int_{\Sigma^{s}_{\lambda}}\frac{1}{|x-y|^{\mu}}s(y)^{2^{\ast}_{\mu}-1}[s(y)-s(y^{\lambda})]dy.
\end{aligned}
\end{equation}

Now by the H\"{o}lder inequality and the HLS inequality, we know that
\begin{equation}\nonumber
\begin{aligned}
\left\|\int_{\Sigma^{s}_{\lambda}}\frac{ts^{{2}^{\ast}_{\mu}-2}(s-s_{\lambda})}{|x-y|^{N-2}}dy\right\|_{L^{2^{\ast}}(\Sigma^{s}_{\lambda})}
&\leq C||ts^{{2}^{\ast}_{\mu}-2}(s-s_{\lambda})||_{L^{\frac{2N}{N+2}}(\Sigma^{s}_{\lambda})}\\
&\leq C\|t\|_{L^{\frac{2N}{\mu}}(\Sigma^{s}_{\lambda})}\|s^{2^{\ast}_{\mu}-2}\|_{L^{\frac{2^{\ast}}{2^{\ast}_{\mu}-2}}(\Sigma^{s}_{\lambda})}||s-s_{\lambda}||_{L^{2^{\ast}}(\Sigma^{s}_{\lambda})}\\
&\leq C\|t\|_{L^{\frac{2N}{\mu}}(\Sigma^{s}_{\lambda})}\|s\|^{2^{\ast}_{\mu}-2}_{L^{2^{\ast}}(\Sigma^{s}_{\lambda})}||s-s_{\lambda}||_{L^{2^{\ast}}(\Sigma^{s}_{\lambda})}.
\end{aligned}
\end{equation}
By inequality \eqref{Inequality2}, we have
\begin{equation}\nonumber
\begin{aligned}
\left\|\int_{\Sigma^{t}_{\lambda}}\frac{s^{{2}^{\ast}_{\mu}-1}(t-t_{\lambda})}{|x-y|^{N-2}}dy\right\|_{L^{2^{\ast}}(\Sigma^{s}_{\lambda})}
&\leq C||s^{{2}^{\ast}_{\mu}-1}(t-t_{\lambda})||_{L^{\frac{2N}{N+2}}(\Sigma^{t}_{\lambda})}\\
&\leq C||s^{{2}^{\ast}_{\mu}-1}||_{L^{\frac{2^{\ast}}{{2}^{\ast}_{\mu}-1}}(\Sigma^{t}_{\lambda})}
||t-t_{\lambda}||_{L^{\frac{2N}{\mu}}(\Sigma^{t}_{\lambda})}\\
&\leq C||s||^{{2}^{\ast}_{\mu}-1}_{L^{2^{\ast}}(\Sigma^{t}_{\lambda})}
||s^{2^{\ast}_{\mu}-1}(s-s_{\lambda})||_{L^{\frac{2^{\ast}}{2^{\ast}_{\mu}}}(\Sigma^{s}_{\lambda})}\\
&\leq C||s||^{{2}^{\ast}_{\mu}-1}_{L^{2^{\ast}}(\Sigma^{t}_{\lambda})}
||s||^{2^{\ast}_{\mu}-1}_{L^{2^{\ast}}(\Sigma^{s}_{\lambda})}||s-s_{\lambda}||_{L^{2^{\ast}}(\Sigma^{s}_{\lambda})}.
\end{aligned}
\end{equation}
Consequently, we know there exists constant $C>0$ such that
\begin{equation}\nonumber
\begin{aligned}
||s-s_{\lambda}||_{L^{2^{\ast}}({\Sigma^{s}_{\lambda}})}&\leq C\left[\left\|\int_{\Sigma^{s}_{\lambda}}\frac{ts^{{2}^{\ast}_{\mu}-2}(s-s_{\lambda})}{|x-y|^{N-2}}dy\right\|_{L^{2^{\ast}}(\Sigma^{s}_{\lambda})}
+\left\|\int_{\Sigma^{t}_{\lambda}}\frac{s^{{2}^{\ast}_{\mu}-1}(t-t_{\lambda})}{|x-y|^{N-2}}dy\right\|_{L^{2^{\ast}}(\Sigma^{s}_{\lambda})}\right]\\
&\leq C\left[\|t\|_{L^{\frac{2N}{\mu}}(\Sigma^{s}_{\lambda})}\|s\|^{2^{\ast}_{\mu}-2}_{L^{2^{\ast}}(\Sigma^{s}_{\lambda})}
+||s||^{{2}^{\ast}_{\mu}-1}_{L^{2^{\ast}}(\Sigma^{t}_{\lambda})}
||s||^{2^{\ast}_{\mu}-1}_{L^{2^{\ast}}(\Sigma^{s}_{\lambda})}\right]||s-s_{\lambda}||_{L^{2^{\ast}}(\Sigma^{s}_{\lambda})}.
\end{aligned}
\end{equation}
\end{proof}
For the case $N>4$, $\mu=4$, one may find that Lemma \ref{Inter2} should be replaced by the following Lemma, the proof is omitted here.
\begin{lem}\label{Inter3}
Suppose that $N>4$ and $\mu=4$, then for any $\lambda>0$ there exists constant $C>0$ such that:
$$
||s-s_{\lambda}||_{L^{2^{\ast}}({\Sigma^{s}_{\lambda}})}\leq C\left[\|t\|_{L^{\frac{N}{2}}(\Sigma^{s}_{\lambda})}
+||s||_{L^{2^{\ast}}(\Sigma^{t}_{\lambda})}
||s||_{L^{2^{\ast}}(\Sigma^{s}_{\lambda})}\right]||s-s_{\lambda}||_{L^{2^{\ast}}(\Sigma^{s}_{\lambda})}.
$$
\end{lem}

The integral inequalities in Lemma \ref{Inter2} and \ref{Inter3} will allow us to begin the procedure of moving plane methods in integral forms.

We first prove that, for $\lambda$ sufficiently positive,
\begin{equation}\label{S8}
s(x)\leq s(x^{\lambda}),~~~t(x)\leq t(x^{\lambda}),~~\forall x\in\Sigma_{\lambda}-\{0^{\lambda}\}.
\end{equation}
Then we can start moving the plane from near $+\infty$ to the left as long as \eqref{S8} holds.

\begin{lem}\label{S1}
 There exists $\lambda_{0}>0$ such that for any $\lambda\geq\lambda_{0}$, $s(x)\leq s(x^{\lambda})$ and $t(x)\leq t(x^{\lambda})$ hold in $\Sigma_{\lambda}-\{0^{\lambda}\}$.
 \end{lem}
\begin{proof}
 Since $s(x)$ and $t(x)$ is integrable, we can choose $\lambda_{0}$ sufficiently large, such that for $\lambda\geq\lambda_{0}$, we have
$$
C\left[\|t\|_{L^{\frac{2N}{\mu}}(\Sigma^{s}_{\lambda})}\|s\|^{2^{\ast}_{\mu}-2}_{L^{2^{\ast}}(\Sigma^{s}_{\lambda})}
+||s||^{{2}^{\ast}_{\mu}-1}_{L^{2^{\ast}}(\Sigma^{t}_{\lambda})}
||s||^{2^{\ast}_{\mu}-1}_{L^{2^{\ast}}(\Sigma^{s}_{\lambda})}\right]<1.
$$
Consequently Lemma \ref{Inter2} implies that
$$
\left\|s-s_{\lambda}\right\|_{L^{2^{\ast}}(\Sigma^{s}_{\lambda})}=0,
$$
from which we can conclude that the set $\Sigma^{s}_{\lambda}$ is empty. Thus we have $s(x)\leq s(x^{\lambda})$ in $\Sigma_{\lambda}-\{0^{\lambda}\}$. From equality \eqref{IEquality}, we can also prove that $t(x)\leq t(x^{\lambda})$.
\end{proof}

Now we begin to move the plane $T_{\lambda}$ to the left as long as \eqref{S8} holds. If we define
$$
\lambda_{1}=\inf\{~\lambda~|~s(x)\leq s(x^{\eta})~,~t(x)\leq t(x^{\eta})~,~x\in\Sigma_{\eta}-\{0^{\eta}\},~\eta\geq\lambda~\},
$$
then we must have $\lambda_{1}<\infty$. Since the plane can be moved from $+\infty$, we show that
 \begin{lem}\label{S2} For any $\lambda_{1}>0$, we have $s(x)\equiv s(x^{\lambda_{1}})$ and $t(x)\equiv t(x^{\lambda_{1}})$ in $\Sigma_{\lambda_{1}}-\{0^{\lambda_{1}}\}$.
 \end{lem}
\begin{proof} Suppose~$s(x)\not\equiv s(x^{\lambda_{1}})$. By \eqref{Minus}, we can obtain that in $\Sigma_{\lambda_{1}}-\{0^{\lambda_{1}}\}$ there holds
$$
s(x)-s(x^{\lambda_{1}})<0.
$$
For any $\eta>0$ small, we can choose $R$ sufficiently large such that
\begin{equation}\label{SSS}
C\left[\|t\|_{L^{\frac{2N}{\mu}}(\mathbb{R}^{N}-B_{R}(0))}\|s\|^{2^{\ast}_{\mu}-2}_{L^{2^{\ast}}(\mathbb{R}^{N}-B_{R}(0))}
+||s||^{{2}^{\ast}_{\mu}-1}_{L^{2^{\ast}}(\mathbb{R}^{N}-B_{R}(0))}
||s||^{2^{\ast}_{\mu}-1}_{L^{2^{\ast}}(\mathbb{R}^{N}-B_{R}(0))}\right]<\eta.
\end{equation}
 Fix $\varepsilon>0$ small, we set
$$
P_{\varepsilon}=\{x\in(\Sigma_{\lambda_{1}}-\{0^{\lambda_{1}}\})\cap B_{R}(0)|s(x^{\lambda_{1}})-s(x)>\varepsilon\}
$$
and
$$
Q_{\varepsilon}=\{x\in(\Sigma_{\lambda_{1}}-\{0^{\lambda_{1}}\})\cap B_{R}(0)|s(x^{\lambda_{1}})-s(x)\leq\varepsilon\}.
$$
For $\lambda>\lambda_{1}$, we fix a narrow domain
$$
\Omega_{\lambda}=(\Sigma_{\lambda}-\{0^{\lambda}\}-\Sigma_{\lambda_{1}})\cap B_{R}(0),
$$
then for any $x\in\Sigma^{s}_{\lambda}\cap P_{\varepsilon}$, we know
$$
s(x^{\lambda_{1}})-s(x^{\lambda})>s(x^{\lambda_{1}})-s(x)>\varepsilon.
$$
However if $\lambda\rightarrow\lambda_{1}$, the Chebyshev inequality implies that
$$
\cl(\Sigma^{s}_{\lambda}\cap P_{\varepsilon})\leq\cl(\{x\in B_{R}(0)-\{0^{\lambda}\}|s(x^{\lambda_{1}})-s(x^{\lambda})>\varepsilon\})\rightarrow0,
$$
where $\cl$ is the Lebesgue measure. Now let $\varepsilon\rightarrow0$ and $\lambda\rightarrow\lambda_{1}$, we get
$$
\cl(\Sigma^{s}_{\lambda}\cap B_{R}(0))\leq\cl(\Sigma^{s}_{\lambda}\cap P_{\varepsilon})+\cl(Q_{\varepsilon})+\cl(\Omega_{\lambda})\rightarrow0.
$$
Combining this fact with \eqref{SSS}, we know there exist $\delta>0$ such that for any $\lambda\in[\lambda_{1}-\delta,\lambda_{1}]$,
$$
C\left[\|t\|_{L^{\frac{2N}{\mu}}(\Sigma^{s}_{\lambda})}\|s\|^{2^{\ast}_{\mu}-2}_{L^{2^{\ast}}(\Sigma^{s}_{\lambda})}
+||s||^{{2}^{\ast}_{\mu}-1}_{L^{2^{\ast}}(\Sigma^{t}_{\lambda})}
||s||^{2^{\ast}_{\mu}-1}_{L^{2^{\ast}}(\Sigma^{s}_{\lambda})}\right]<1.
$$
Thus we assert that for all $\lambda\in[\lambda_{1}-\delta,\lambda_{1}]$,
$$
s(x)\leq s(x^{\lambda}),~~~t(x)\leq t(x^{\lambda}),~~\forall x\in\Sigma_{\lambda}-\{0^{\lambda}\}.
$$
This contradicts with the definition of $\lambda_{1}$, therefore \eqref{IEquality} implies $t(x)\equiv t(x^{\lambda_{1}})$.
\end{proof}
Similarly, the plane can also be moved from $-\infty$ to right, therefore we denote the corresponding parameter $\lambda_{1}^{\prime}$ satisfying
$$
\lambda^{\prime}_{1}=\sup\{~\lambda~|~s(x)\leq s(x^{\eta})~,~t(x)\leq t(x^{\eta})~,~x\in\Sigma^{\prime}_{\eta}-\{0^{\eta}\},~\eta\leq\lambda~\}.
$$
Where $\Sigma^{\prime}_{\lambda}=\{x\in\mathbb{R}^{N}:x_{1}<\lambda\}$. If $\lambda^{\prime}_{1}<0$, $s(x)$ and $t(x)$ are also radially symmetric about $T_{\lambda^{\prime}_{1}}$ in the similar argument.

When $\lambda_{1}=\lambda^{\prime}_{1}=0$, then $s(x)$ and $t(x)$ are radially symmetric about $T_{0}$, if it holds for every direction, then $s(x)$ and $t(x)$ must be radially symmetric about origin, that is $u(x)$ and $v(x)$ are radially symmetric about origin.

Otherwise, there exists some direction, might as well choose $x_{1}$ direction, such that $\lambda_{1}=\lambda^{\prime}_{1}\neq0$, then Lemma \ref{S2} implies that $s(x)$ and $t(x)$ are radially symmetric about $T_{\lambda_{1}}$, this implies $s(x)$ and $t(x)$ have no singularity at origin, hence $u(x)$ and $v(x)$ approach to 0 at infinity, then we can use the moving plane methods directly on $u(x)$ and $v(x)$ to get the radial symmetry.

\section{Uniqueness}

In order to prove the uniqueness, we may show that the solutions are invariable under the translation transform, homothetic transform and Kelvin-type transform.
\begin{lem}\label{U1}
Suppose $a\in\mathbb{R}^{N}$, $k\in\mathbb{R}$, $u$ is the positive solution of \eqref{OPQ}, then $u(x+a)$, $k^{\frac{N-2}{2}}u(kx)$ and $\frac{1}{|x|^{N-2}}u(\frac{x}{|x|^{2}})$ are still the positive solutions of equation \eqref{OPQ}.
\end{lem}
\begin{proof} We split the proof into three steps.

Step 1. Assume that $y=y^{\prime}+a$, $z=z^{\prime}+a$, then
\begin{equation}\nonumber
\begin{aligned}
u(x+a)
&=\int_{\mathbb{R}^{N}}\frac{1}{|x+a-y|^{N-2}}\left(\int_{\mathbb{R}^{N}}\frac{u(z)^{2^{\ast}_{\mu}}}{|y-z|^{\mu}}dz\right)u(y)^{{2}^{\ast}_{\mu}-1}dy\\
&=\int_{\mathbb{R}^{N}}\frac{1}{|x+a-(y^{\prime}+a)|^{N-2}}\left(\int_{\mathbb{R}^{N}}\frac{u(z^{\prime}+a)^{2^{\ast}_{\mu}}}{|y^{\prime}+a-(z^{\prime}+a)|^{\mu}}dz^{\prime}\right)u(y^{\prime}+a)^{{2}^{\ast}_{\mu}-1}dy^{\prime}\\
&=\int_{\mathbb{R}^{N}}\frac{1}{|x-y^{\prime}|^{N-2}}\left(\int_{\mathbb{R}^{N}}\frac{u(z^{\prime}+a)^{2^{\ast}_{\mu}}}{|y^{\prime}-z^{\prime}|^{\mu}}dz^{\prime}\right)u(y^{\prime}+a)^{{2}^{\ast}_{\mu}-1}dy^{\prime}.
\end{aligned}
\end{equation}
It is obviously that $u(x+a)$ is solution of \eqref{OPQ}.

Step 2. Assume that $y=ky^{\prime}$, $z=kz^{\prime}$, then
\begin{equation}\nonumber
\begin{aligned}
k^{\frac{N-2}{2}}u(kx)
&=k^{\frac{N-2}{2}}\int_{\mathbb{R}^{N}}\frac{1}{|kx-y|^{N-2}}\left(\int_{\mathbb{R}^{N}}\frac{u(z)^{2^{\ast}_{\mu}}}{|y-z|^{\mu}}dz\right)u(y)^{{2}^{\ast}_{\mu}-1}dy\\
&=k^{\frac{N-2}{2}}\int_{\mathbb{R}^{N}}\frac{1}{|kx-ky^{\prime}|^{N-2}}\left(\int_{\mathbb{R}^{N}}\frac{u(kz^{\prime})^{2^{\ast}_{\mu}}}{|ky^{\prime}-kz^{\prime}|^{\mu}}k^{N}dz^{\prime}\right)u(ky^{\prime})^{{2}^{\ast}_{\mu}-1}k^{N}dy^{\prime}\\
&=\int_{\mathbb{R}^{N}}\frac{1}{|x-y^{\prime}|^{N-2}}\left(\int_{\mathbb{R}^{N}}\frac{\left[k^{\frac{N-2}{2}}u(kz^{\prime})\right]^{2^{\ast}_{\mu}}}{|y^{\prime}-z^{\prime}|^{\mu}}dz^{\prime}\right)\left[k^{\frac{N-2}{2}}u(ky^{\prime})\right]^{{2}^{\ast}_{\mu}-1}dy^{\prime}.
\end{aligned}
\end{equation}
This implies $k^{\frac{N-2}{2}}u(kx)$ is also a solution of \eqref{OPQ}.

Step 3. Assume that $y=\frac{y^{\prime}}{|y^{\prime}|^{2}}$, $z=\frac{z^{\prime}}{|z^{\prime}|^{2}}$, then
\begin{equation}\nonumber
\begin{aligned}
\frac{1}{|x|^{N-2}}u\left(\frac{x}{|x|^{2}}\right)
&=\frac{1}{|x|^{N-2}}\int_{\mathbb{R}^{N}}\frac{1}{|\frac{x}{|x|^{2}}-y|^{N-2}}\left(\int_{\mathbb{R}^{N}}\frac{u(z)^{2^{\ast}_{\mu}}}{|y-z|^{\mu}}dz\right)u(y)^{{2}^{\ast}_{\mu}-1}dy\\
&=\frac{1}{|x|^{N-2}}\int_{\mathbb{R}^{N}}\frac{1}{\left|\frac{x}{|x|^{2}}-\frac{y^{\prime}}{|y^{\prime}|^{2}}\right|^{N-2}}\left(\int_{\mathbb{R}^{N}}\frac{u\left(\frac{z^{\prime}}{|z^{\prime}|^{2}}\right)^{2^{\ast}_{\mu}}}{\left|\frac{y^{\prime}}{|y^{\prime}|^{2}}-\frac{z^{\prime}}{|z^{\prime}|^{2}}\right|^{\mu}}\frac{1}{|z^{\prime}|^{2N}}dz^{\prime}\right)u\left(\frac{y^{\prime}}{|y^{\prime}|^{2}}\right)^{{2}^{\ast}_{\mu}-1}\frac{1}{|y^{\prime}|^{2N}}dy^{\prime}\\
&=\int_{\mathbb{R}^{N}}\frac{1}{\left|x-y^{\prime}\right|^{N-2}}\left(\int_{\mathbb{R}^{N}}\frac{\left[\frac{1}{|z^{\prime}|^{N-2}}u\left(\frac{z^{\prime}}{|z^{\prime}|^{2}}\right)\right]^{2^{\ast}_{\mu}}}{\left|y^{\prime}-z^{\prime}\right|^{\mu}}dz^{\prime}\right)\left[\frac{1}{|y^{\prime}|^{N-2}}u\left(\frac{y^{\prime}}{|y^{\prime}|^{2}}\right)\right]^{{2}^{\ast}_{\mu}-1}dy^{\prime}.
\end{aligned}
\end{equation}
Thus we conclude that $\frac{1}{|x|^{N-2}}u(\frac{x}{|x|^{2}})$ is still a solution of \eqref{OPQ}.
 \end{proof}

{\flushleft{\bf Proof of Theorem \ref{TH} completed.}} Since it was proved that the solution is radially symmetric, next we are going to show that the solution assumes the form \eqref{class}. Now we will show that the positive solution behaves at infinity as
$$
u(x)=O\left(\frac{1}{|x|^{N-2}}\right).
$$
Otherwise, we suppose that the contrary holds. For any $x_{0}\in\mathbb{R}^{N}$, we take the Kelvin-type transform centered at $x_{0}$ by
$$
s_{0}(x)=\frac{1}{|x-x_{0}|^{N-2}}u\left(\frac{x-x_{0}}{|x-x_{0}|^{2}}+x_{0}\right).
$$
When $x\rightarrow x_{0}$, we realize $x_{0}$ must be singularity of $s_{0}$. Based on the arguments above, we conclude that $s_{0}$ must be radially symmetric about $x_{0}$, thus we reach the conclusion that $u(x)$ must be constant, however, that is impossible.

Define
$$
u_{\infty}=\lim_{x\rightarrow\infty}|x|^{N-2}u(x),
$$
clearly, $u_{\infty}$ is significative. We choose proper $k$ that satisfies
$$
k=\left(\frac{u_{\infty}}{u(0)}\right)^{\frac{1}{N-2}}
$$
and define
$$
u_{k}(x)=k^{\frac{N-2}{2}}u(kx).
$$
Suppose the given unit vector $e=(1,0,...,0)$ and
$$
s_{k}(x)=\frac{1}{|x|^{N-2}}u_{k}\left(\frac{x}{|x|^{2}}-e\right).
$$
From lemma \ref{U1}, we know $s_k(x)$ is a solution of \eqref{OPQ}.

From the simple calculation we can get
$$
s_{k}(e)=u_{k}(0)=k^{\frac{N-2}{2}}u(0).
$$
On the other hand, since
\begin{equation}\nonumber
\begin{aligned}
s_{k}(0)&=\lim_{x\rightarrow0}\frac{1}{|x|^{N-2}}u_{k}\left(\frac{x}{|x|^{2}}-e\right)\\
&=\lim_{x\rightarrow0}\frac{1}{|x|^{N-2}}k^{\frac{N-2}{2}}u\left[k\left(\frac{x}{|x|^{2}}-e\right)\right]\\
&=k^{\frac{2-N}{2}}\lim_{x\rightarrow\infty}|kx|^{N-2}u\left[k(x-e)\right]\\
&=k^{\frac{2-N}{2}}u_{\infty},
\end{aligned}
\end{equation}
we know $s_{k}(0)=s_{k}(e)$. Since $s_{k}(x)$ is radially symmetric, hence $s_{k}(x)$ must be symmetric about $\frac{e}{2}$.

Furthermore, we are able to map the point on $x_{1}=\frac{1}{2}$ to the unit centered at $e$. Let
$$
x^{\prime}=\left\{x\in\mathbb{R}^{N}|x_{1}=\frac{1}{2}\right\}
$$
then
$$
\left|\frac{x^{\prime}}{|x^{\prime}|^{2}}-e\right|^{2}=\frac{1-2e x^{\prime}}{|x^{\prime}|^{2}}+1=1,
$$
which implies that
$$
s_{k}(x^{\prime})=\frac{1}{|x^{\prime}|^{N-2}}u_{k}\left(\frac{x^{\prime}}{|x^{\prime}|^{2}}-e\right)=\frac{c}{|x^{\prime}|^{N-2}},
$$
where $c$ denote the value of $u_{k}(x)$ on unit centered at origin. Thus we can conclude that for any $x\in\mathbb{R}^{N}$,
$$
s_{k}(x)=\frac{c}{(\frac{1}{4}+|x-\frac{e}{2}|^{2})^\frac{{N-2}}{2}}.
$$
Apparently $s_{k}(x)$ satisfies the form \eqref{class}. Since $s_{k}(x)$ is the composite transformation of $u(x)$, then $u(x)$ also assumes the form \eqref{class}. $\hfill{} \Box$

\section{Nondegeneracy}

We introduce the following equivalent form of the HLS inequality with Riesz potential.
\begin{Prop}\label{Prop}
Let $1\leq r<s<\infty$ and $0<\mu<N$ satisfy
$$
\frac{1}{r}-\frac{1}{s}=\frac{N-\mu}{N}.
$$
Then for $\mu$ sufficient close to $N$, there exists constant $C(N,r)>0$ such that for any $f\in L^{r}(\mathbb{R}^{N})$, there holds
\begin{equation}\label{RFV}
||I_{\mu}\ast f||_{L^{s}(\mathbb{R}^{N})}\leq C(N,r)||f||_{L^{r}(\mathbb{R}^{N})}.
\end{equation}
\end{Prop}
\begin{proof}
Due to the equivalent form of Hardy-littlewood-Sobolev inequality, we have
$$
||\frac{1}{|\cdot|^{\mu}}\ast f||_{L^{s}(\mathbb{R}^{N})}\leq K(N,\mu,r)||f||_{L^{r}(\mathbb{R}^{N})}.
$$
Where $K(n,\mu,r)$ satisfies
$$
\limsup_{\mu\rightarrow N}(N-\mu)K(N,\mu,r)\leq\frac{2}{r(r-1)}|\mathbb{S}^{N-1}|.
$$
Then \eqref{RFV} holds because of $
\Gamma(\frac{N-\mu}{2})\sim\frac{1}{N-\mu}$ as $\mu\rightarrow N$.
\end{proof}
Notice the assumptions in Theorem \ref{TH} that $0<\mu<N$, if $N=3\ \hbox{or} \ 4$ and $0<\mu\leq4$ if $N\geq5$, we can only prove the nondegeneracy of $U_\mu(x)$ to \eqref{cc} when $\mu$ is close to  $N=3\ \hbox{or} \ 4$. And here we only give the proof The corresponding Euler-Lagrange functional of \eqref{cc} is
$$
I_{\mu}(u)=\frac{1}{2}\int_{\mathbb{R}^{3}}|\nabla u|^{2}dx-\frac{1}{2\cdot2^{\ast}_{\mu}}\int_{\mathbb{R}^{3}}(I_{\mu}\ast u^{2^{\ast}_{\mu}})u^{{2}^{\ast}_{\mu}}dx.
$$
and the differential of $I_{\mu}(u)$ is
$$
<I^{\prime}_{\mu}(u),\nu>=\int_{\mathbb{R}^{3}}\nabla u\nabla \nu dx-\int_{\mathbb{R}^{3}}(I_{\mu}\ast u^{2^{\ast}_{\mu}})u^{{2}^{\ast}_{\mu}-1}\nu dx,~~\nu\in D^{1,2}(\mathbb{R}^{3}).
$$
Due to Remark \ref{TH2}, the functional $I_{\mu}(u)$ possesses the family of critical points, depending on 4-parameters $\xi\in\mathbb{R}^{3}$ and $t\in\mathbb{R}^{+}$,
$$
z_{t,\xi,\mu}=S^{\frac{(N-\mu)(2-N)}{4(N-\mu+2)}}C^{\ast}(N,\mu)^{\frac{2-N}{2(N-\mu+2)}}[\frac{\sqrt{N(N-2)}t}{(t^{2}+|x-\xi|^{2})}]^{\frac{N-2}{2}}.
$$
Then we define $Z_{\mu}$ is the 4-dimensional critical manifold of $I_{\mu}(u)$ which satisfies
$$
Z_{\mu}=\{z_{t,\xi,\mu}|t>0,\xi\in\mathbb{R}^{3}\},
$$
and
$T_{z}Z_{\mu}$ is the tangent space to $Z_{\mu}$. We have
$$
<I^{\prime\prime}_{\mu}(z)[v],\phi>=0,~~\forall v\in T_{z}Z_{\mu},~~\forall\phi\in D^{1,2}(\mathbb{R}^{3}).
$$

If $\mu$ is close to $0$ or $N$, the nondegeneracy of the ground states of the subcritical Chouquard equation was studied in \cite{SJ}. Inspired by \cite{SJ}, we are going to prove that $Z_{\mu}$ satisfies the nondegeneracy condition when $\mu\rightarrow 3$.

{\bf Proof of Theorem \ref{TH3}.}
Fix $t, \xi$, since any $U_\mu\in Z_{\mu}$ satisfies \eqref{cc}, we can conclude that any $\partial_{t}U_\mu=\frac{\partial U_\mu}{\partial t}$, $\partial_{i}U_\mu=\frac{\partial U_\mu}{\partial x_{i}}\in T_{z}{Z_{\mu}}(i=1,2,3)$ satisfy the following equation:
\begin{equation}\label{A}
A_{\mu}(\psi)=-\Delta\psi-2^{\ast}_{\mu}(I_{\mu}\ast(U_\mu^{2^{\ast}_{\mu}-1}\psi))U_\mu^{{2}^{\ast}_{\mu}-1}-(2^{\ast}_{\mu}-1)(I_{\mu}\ast U_\mu^{2^{\ast}_{\mu}})U_\mu^{{2}^{\ast}_{\mu}-2}\psi=0,~~\psi\in L^{2^{\ast}}(\mathbb{R}^{3}).
\end{equation}
It is obvious that $T_{z}{Z_{\mu}}\subseteq\mathbf{Ker}{[I_{\mu}^{\prime\prime}(z)]}$, next we show that $\mathbf{Ker}{[I_{\mu}^{\prime\prime}(z)]}\subseteq T_{z}{Z_{\mu}}$. Noting that $U_\mu=cU_0$ and recalling the finite dimensional vector space
$$
T_{z}{Z_{\mu}}:=\mathbf{span}\{\partial_{1}U_\mu,\partial_{2}U_\mu,\partial_{3}U_\mu,\partial_{t}U_\mu\}=\mathbf{span}\{\partial_{1}U_0,\partial_{2}U_0,\partial_{3}U_0,\partial_{t}U_0\},
$$
On the contrary, we suppose that there exists a sequence $\{\mu_{n}\}$ with $\mu_{n}\rightarrow 3$ as $n\rightarrow\infty$ and for each $\mu_{n}$ we have nontrivial solution $\psi_{n}$ of \eqref{A} in the topological complement of $T_{z}{Z_{\mu}}$ in $L^{2^*}({\mathbb{R}^{3}})$.

We define the functional
$$
L[\psi]=2^{\ast}_{\mu}(I_{\mu}\ast(U_\mu^{2^{\ast}_{\mu}-1}\psi))U_\mu^{{2}^{\ast}_{\mu}-1}+(2^{\ast}_{\mu}-1)(I_{\mu}\ast U_\mu^{2^{\ast}_{\mu}})U_\mu^{{2}^{\ast}_{\mu}-2}\psi,
$$
then for any $\varphi\in D^{1,2}(\mathbb{R}^{3})$, H$\ddot{o}$lder inequality and Proposition \ref{Prop} are applied to imply
\begin{equation}\nonumber
\begin{aligned}
|<L[\psi],\varphi>|&=|2^{\ast}_{\mu}\int_{\mathbb{R}^{3}}(I_{\mu}\ast(U_\mu^{2^{\ast}_{\mu}-1}\psi))U_\mu^{{2}^{\ast}_{\mu}-1}\varphi dx+(2^{\ast}_{\mu}-1)\int_{\mathbb{R}^{3}}(I_{\mu}\ast U_\mu^{2^{\ast}_{\mu}})U_\mu^{{2}^{\ast}_{\mu}-2}\psi\varphi dx|\\
&\leq2^{\ast}_{\mu}|\int_{\mathbb{R}^{3}}(I_{\mu}\ast(U_\mu^{2^{\ast}_{\mu}-1}\psi))U_\mu^{{2}^{\ast}_{\mu}-1}\varphi dx|+(2^{\ast}_{\mu}-1)|\int_{\mathbb{R}^{3}}(I_{\mu}\ast U_\mu^{2^{\ast}_{\mu}})U_\mu^{{2}^{\ast}_{\mu}-2}\psi\varphi dx|\\
&\leq C||U_\mu^{{2}^{\ast}_{\mu}-1}||_{L^{\frac{3}{2}}}||U_\mu^{{2}^{\ast}_{\mu}-1}||_{\frac{3}{3-\mu}}||\psi||_{L^{2^{\ast}}}||\varphi||_{L^{2^{\ast}}}+||U_\mu^{{2}^{\ast}_{\mu}}||_{L^{\frac{3}{4-\mu}}}||U_\mu^{{2}^{\ast}_{\mu}-2}||_{L^{3}}||\psi||_{L^{2^{\ast}}}||\varphi||_{L^{2^{\ast}}}\\
&=C[||U_\mu||^{{2}^{\ast}_{\mu}-1}_{L^{\frac{3(5-\mu)}{2}}}||U_\mu||^{{2}^{\ast}_{\mu}-1}_{L^{\frac{3(5-\mu)}{3-\mu}}}+||U_\mu||^{{2}^{\ast}_{\mu}}_{L^{\frac{3(6-\mu)}{4-\mu}}}||U_\mu||^{{2}^{\ast}_{\mu}-2}_{L^{3(4-\mu)}}]||\psi||_{L^{2^{\ast}}}||\varphi||_{D^{1,2}}.
\end{aligned}
\end{equation}
Where we use the regularity of $U_\mu\in L^{p}(\mathbb{R}^{3}), (3<p\leq\infty)$ and $\psi\in L^{2^{\ast}}(\mathbb{R}^{3})$. Therefore we find that the functional $L[\psi]\in(D^{1,2}(\mathbb{R}^{3}))^{\ast}$ where $(D^{1,2}(\mathbb{R}^{3}))^{\ast}$ denote the dual space of $D^{1,2}(\mathbb{R}^{3})$. Since $-\Delta\psi\in(D^{1,2}(\mathbb{R}^{3}))^{\ast}$, then we achieve that $\psi\in D^{1,2}(\mathbb{R}^{3})$.

Now we may assume that $\psi_{n}$ is a sequence of unit solution for the linearized equation at $U_{n}:=U_{\mu_n}$, hence there exists $\psi_{0}\in D^{1,2}(\mathbb{R}^{3})$, such that $\psi_{n}\rightharpoonup\psi_{0}$ in $D^{1,2}(\mathbb{R}^{3})$ as $n\rightarrow\infty$. Consequently for any $\varphi\in D^{1,2}(\mathbb{R}^{3})$,
\begin{equation}\label{B}
\begin{aligned}
\int_{\mathbb{R}^{3}}\nabla\psi_n\nabla\varphi dx=2^{\ast}_{\mu_n}\int_{\mathbb{R}^{3}}\big(I_{\mu_{n}}\ast(U^{2^{\ast}_{\mu_{n}}-1}_{n}\psi_n)\big)U^{2^{\ast}_{\mu_{n}}-1}_{n}\varphi dx+(2^{\ast}_{\mu_n}-1)\int_{\mathbb{R}^{3}}\big(I_{\mu_{n}}\ast(U^{2^{\ast}_{\mu_{n}}}_n)\big)U^{2^{\ast}_{\mu_{n}}-2}_{n}\psi_n\varphi dx.
\end{aligned}
\end{equation}

We claim that
\begin{equation}\label{Claim1}
\int_{\mathbb{R}^{3}}(I_{\mu_{n}}\ast(U^{2^{\ast}_{\mu_{n}}-1}_{n}\psi_{n}))U^{2^{\ast}_{\mu_{n}}-1}_{n}\varphi dx\rightarrow\int_{\mathbb{R}^{3}}U_{0}^{4}\psi_{0}\varphi dx, \ \ \hbox{as} \ \ n\to \infty.
\end{equation}
In fact, notice that
$$
I_{\mu_{n}}\ast(U^{2^{\ast}_{\mu_{n}}-1}_{n}\psi_{n})-U^{2^{\ast}_{\mu_{n}}-1}_{n}\psi_{n}
=(I_{\mu_{n}}\ast(U^{2^{\ast}_{\mu_{n}}-1}_{n}\psi_{n}-U^{2}_{0}\psi_{0}))
+(I_{\mu_{n}}\ast(U^{2}_{0}\psi_{0})-U^{2}_{0}\psi_{0})
+(U^{2}_{0}\psi_{0}-U^{2^{\ast}_{\mu_{n}}-1}_{n}\psi_{n}),
$$
then we can estimate the integral in three parts.

First, for $0<\varepsilon<1$ small enough, we have
\begin{equation}\nonumber
\begin{aligned}
&\int_{\mathbb{R}^{3}}\big(I_{\mu_{n}}\ast(U^{2^{\ast}_{\mu_{n}}-1}_{n}\psi_{n}-U^{2}_{0}\psi_{0})\big)U^{2^{\ast}_{\mu_{n}}-1}_{n}\varphi dx\\
&\hspace{6mm}\leq||I_{\mu_{n}}\ast(U^{2^{\ast}_{\mu_{n}}-1}_{n}\psi_{n}-U^{2}_{0}\psi_{0})||_{L^{\frac{6}{2\mu_{n}-5+2\varepsilon}}}
||U^{2^{\ast}_{\mu_{n}}-1}_{n}||_{L^{\frac{3}{5-\mu_{n}-\varepsilon}}}||\varphi||_{L^{2^{\ast}}}\\
&\hspace{6mm}\leq C||U^{2^{\ast}_{\mu_{n}}-1}_{n}\psi_{n}-U^{2}_{0}\psi_{0}||_{L^{\frac{6}{1+2\varepsilon}}}
||U_{n}||^{{2^{\ast}_{\mu_{n}}-1}}_{L^{\frac{3(5-\mu_{n})}{(5-\mu_{n}-\varepsilon)}}}||\varphi||_{L^{2^{\ast}}}.
\end{aligned}
\end{equation}
By the decay estimates of $U_0$ and the local compact embedding, Remark \ref{TH2} and the Dominated Convergence Theorem imply that
\begin{equation}\nonumber
\begin{aligned}
&||U^{2^{\ast}_{\mu_{n}}-1}_{n}\psi_{n}-U^{2}_{0}\psi_{0}||_{L^{\frac{6}{1+2\varepsilon}}}\\
&\hspace{6mm}\leq||U^{2^{\ast}_{\mu_{n}}-1}_{n}\psi_{n}-U^{2}_{0}\psi_{n}||_{L^{\frac{6}{1+2\varepsilon}}}
+||U^{2}_{0}\psi_{n}-U^{2}_{0}\psi_{0}||_{L^{\frac{6}{1+2\varepsilon}}}\\
&\hspace{6mm}\leq||U^{2^{\ast}_{\mu_{n}}-1}_{n}-U^{2}_{0}||_{L^{\frac{3}{\varepsilon}}}||\psi_{n}||_{L^{6}}
+||U^{2}_{0}\psi_{n}-U^{2}_{0}\psi_{0}||_{L^{\frac{6}{1+2\varepsilon}}}\rightarrow0,\ \ \hbox{as}\ \ n\to\infty,
\end{aligned}
\end{equation}
consequently,
$$
\int_{\mathbb{R}^{3}}(I_{\mu_{n}}\ast(U^{2^{\ast}_{\mu_{n}}-1}_{n}\psi_{n}-U^{2}_{0}\psi_{0}))U^{2^{\ast}_{\mu_{n}}-1}_{n}\varphi dx\to 0,\ \ \hbox{as}\ \ n\to\infty.
$$

Second, notice that $||U^{2}_{0}\psi_{0}||_{L^{2}}\leq||U_{0}||^{2}_{L^{6}}||\psi_{0}||_{L^{6}}$ and
\begin{equation}\nonumber
\begin{aligned}
||I_{\mu_{n}}\ast(U^{2}_{0}\psi_{0})||_{L^{2}}\leq||U^{2}_{0}\psi_{0}||_{L^{\frac{6}{9-2\mu_{n}}}}
\leq||U_{0}||^{2}_{L^{\frac{6}{4-\mu_{n}}}}||\psi_{0}||_{L^{6}},
\end{aligned}
\end{equation}
thus $U^{2}_{0}\psi_{0}$ and $I_{\mu_{n}}\ast(U^{2}_{0}\psi_{0})$ belong to $L^{2}(\mathbb{R}^{3})$.
We can use Fourier transform formula to get
$$
||I_{\mu_{n}}\ast(U^{2}_{0}\psi_{0})-U^{2}_{0}\psi_{0}||_{L^{2}}
=||\widehat{I_{\mu_{n}}\ast(U^{2}_{0}\psi_{0})}-\widehat{U^{2}_{0}\psi_{0}}||_{L^{2}}
=||\left(\frac{1}{(2\pi|\xi|)^{3-\mu_{n}}}-1\right)\widehat{U^{2}_{0}\psi_{0}}||_{L^{2}}.
$$
Decomposing $\mathbb{R}^{3}=B_{1}(0)\cup \{\mathbb{R}^{3}-B_{1}(0)\}$, by the decay estimates of $U_0$ and the Dominated Convergence Theorem, we conclude that
$$
||I_{\mu_{n}}\ast(U^{2}_{0}\psi_{0})-U^{2}_{0}\psi_{0}||_{L^{2}}\to 0,\ \ \hbox{as}\ \ n\to\infty.
$$
Now by H\"{o}lder inequality we know
\begin{equation}\nonumber
\begin{aligned}
&\int_{\mathbb{R}^{3}}\big(I_{\mu_{n}}\ast(U^{2}_{0}\psi_{0})-U^{2}_{0}\psi_{0}\big)U^{2^{\ast}_{\mu_{n}}-1}_{n}\varphi dx\\
&\hspace{6mm}\leq||I_{\mu_{n}}\ast(U^{2}_{0}\psi_{0})-U^{2}_{0}\psi_{0}||_{L^{2}}
||U_{n}||^{2^{\ast}_{\mu_{n}}-1}_{L^{3(5-\mu_{n})}}||\varphi||_{L^{2^{\ast}}}\to 0,\ \ \hbox{as}\ \ n\to\infty.
\end{aligned}
\end{equation}

Third, as in the first step, for $0<\varepsilon<1$ small enough,
\begin{equation}\nonumber
\begin{aligned}
&\int_{\mathbb{R}^{3}}(U^{2}_{0}\psi_{0}-U^{2^{\ast}_{\mu_{n}}-1}_{n}\psi_{n})U^{2^{\ast}_{\mu_{n}}-1}_{n}\varphi dx\\
&\hspace{6mm}\leq||U^{2}_{0}\psi_{0}-U^{2^{\ast}_{\mu_{n}}-1}_{n}\psi_{n}||_{L^{\frac{6}{1+2\varepsilon}}}
||U_{n}||^{2^{\ast}_{\mu_{n}}-1}_{L^{\frac{3(6-\mu_{n})}{2-\varepsilon}}}||\varphi||_{L^{2^{\ast}}}\rightarrow0.\ \ \hbox{as}\ \ n\to\infty.
\end{aligned}
\end{equation}
Therefore, as $n\rightarrow\infty$, we obtain that
$$
\int_{\mathbb{R}^{3}}(I_{\mu_{n}}\ast(U^{2^{\ast}_{\mu_{n}}-1}_{n}\psi_{n}))U^{2^{\ast}_{\mu_{n}}-1}_{n}\varphi dx-\int_{\mathbb{R}^{3}}U^{2^{\ast}_{\mu_{n}}-1}_{n}\psi_{n}U^{2^{\ast}_{\mu_{n}}-1}_{n}\varphi dx\to0.
$$
Meanwhile, since $U_n\to U_0$ as $n\to\infty$, we have
$$
\int_{\mathbb{R}^{3}}U^{2^{\ast}_{\mu_{n}}-1}_{n}\psi_{n}U^{2^{\ast}_{\mu_{n}}-1}_{n}\varphi dx\rightarrow\int_{\mathbb{R}^{3}}U_{0}^{4}\psi_{0}\varphi dx.
$$
Consequently, we obtain Claim \eqref{Claim1}  that
$$
\int_{\mathbb{R}^{3}}\big(I_{\mu_{n}}\ast(U^{2^{\ast}_{\mu_{n}}-1}_{n}\psi_{n})\big)U^{2^{\ast}_{\mu_{n}}-1}_{n}\varphi  dx\rightarrow\int_{\mathbb{R}^{3}}U_{0}^{4}\psi_{0}\varphi dx, \ \ \hbox{as} \ \ n\to \infty.
$$

Similarly, we claim that
\begin{equation}\label{Claim2}
\int_{\mathbb{R}^{3}}(I_{\mu_{n}}\ast(U^{2^{\ast}_{\mu_{n}}}_{n}))U^{2^{\ast}_{\mu_{n}}-2}_{n}\psi_{n}\varphi dx\rightarrow\int_{\mathbb{R}^{3}}U_{0}^{4}\psi_{0}\varphi dx, \ \ \hbox{as} \ \ n\to \infty.
\end{equation}
In fact, notice that
$$
I_{\mu_{n}}\ast(U^{2^{\ast}_{\mu_{n}}}_{n})-U^{2^{\ast}_{\mu_{n}}}_{n}
=\big(I_{\mu_{n}}\ast(U^{2^{\ast}_{\mu_{n}}}_{n}-U^{3}_{0})\big)
+(I_{\mu_{n}}\ast(U^{3}_{0})-U^{3}_{0})
+(U^{3}_{0}-U^{2^{\ast}_{\mu_{n}}}_{n}),
$$
then we can estimate the integral in three parts.

First, for $0<\varepsilon<1$ small, as $n\to \infty$, we have
\begin{equation}\nonumber
\begin{aligned}
&\int_{\mathbb{R}^{3}}(I_{\mu_{n}}\ast(U^{2^{\ast}_{\mu_{n}}}_{n}-U^{3}_{0}))U^{2^{\ast}_{\mu_{n}}-2}_{n}\psi_{n}\varphi dx\\
&\hspace{6mm}\leq||I_{\mu_{n}}\ast(U^{2^{\ast}_{\mu_{n}}}_{n}-U^{3}_{0})||_{L^{\frac{3}{\mu_{n}-2+\varepsilon}}}
||U^{2^{\ast}_{\mu_{n}}-2}_{n}||_{L^{\frac{3}{4-\mu_{n}-\varepsilon}}}||\psi_{n}||_{L^{2^{\ast}}}||\varphi||_{L^{2^{\ast}}}\\
&\hspace{6mm}\leq C||U^{2^{\ast}_{\mu_{n}}}_{n}-U^{3}_{0}||_{L^{\frac{3}{1+\varepsilon}}}
||U_{n}||^{2^{\ast}_{\mu_{n}}-2}_{L^{\frac{3(4-\mu_{n})}{4-\mu_{n}-\varepsilon}}}||\psi_{n}||_{L^{2^{\ast}}}||\varphi ||_{L^{2^{\ast}}}\rightarrow0,
\end{aligned}
\end{equation}
where we apply Remark \ref{TH2} and the Dominated Convergence Theorem.

Second, using the Fourier transform arguments again, we know
$$
\lim\limits_{\mu_{n}\rightarrow3}||I_{\mu_{n}}\ast(U^{3}_{0})-U^{3}_{0}||_{L^{2}}=0.
$$
Thus, by H\"{o}lder inequality, we have
\begin{equation}\nonumber
\begin{aligned}
&\int_{\mathbb{R}^{3}}(I_{\mu_{n}}\ast(U^{3}_{0})-U^{3}_{0})U^{2^{\ast}_{\mu_{n}}-2}_{n}\psi_{n}\varphi dx\\
&\hspace{6mm}\leq||I_{\mu_{n}}\ast(U^{3}_{0})-U^{3}_{0}||_{L^{2}}
||U_{n}||^{2^{\ast}_{\mu_{n}}-2}_{L^{6(4-\mu_{n})}}||\psi_{n}||_{L^{2^{\ast}}}||\varphi ||_{L^{2^{\ast}}}\rightarrow0, \ \ \hbox{as} \ \ n\to \infty.
\end{aligned}
\end{equation}

Third, for $0<\varepsilon<1$ small enough,
\begin{equation}\nonumber
\begin{aligned}
&\int_{\mathbb{R}^{3}}(U^{3}_{0}-U^{2^{\ast}_{\mu_{n}}}_{n})U^{2^{\ast}_{\mu_{n}}-2}_{n}\psi_{n}\varphi  dx\\
&\hspace{6mm}\leq||U^{3}_{0}-U^{2^{\ast}_{\mu_{n}}}_{n}||_{L^{\frac{3}{1+\varepsilon}}}
||U_{n}||^{2^{\ast}_{\mu_{n}}-2}_{L^{\frac{3(4-\mu_{n})}{1-\varepsilon}}}||\psi_{n}||_{L^{2^{\ast}}}||\varphi ||_{L^{2^{\ast}}}\rightarrow0, \ \ \hbox{as} \ \ n\to \infty.
\end{aligned}
\end{equation}
Therefore, as $n\rightarrow\infty$,
$$
\int_{\mathbb{R}^{3}}(I_{\mu_{n}}\ast(U^{2^{\ast}_{n}}_{n}))U^{2^{\ast}_{\mu_{n}}-2}_{n}\psi_{n}\varphi  dx-\int_{\mathbb{R}^{3}}U^{2^{\ast}_{\mu_{n}}}_{n}U^{2^{\ast}_{\mu_{n}}-2}_{n}\psi_{n}\varphi dx\to0,
$$
Notice that
$$
\int_{\mathbb{R}^{3}}U^{2^{\ast}_{\mu_{n}}}_{n}U^{2^{\ast}_{\mu_{n}}-2}_{n}\psi_{n}\varphi dx\rightarrow\int_{\mathbb{R}^{3}}U_{0}^{4}\psi_{0}\varphi dx, \ \ \hbox{as} \ \ n\to \infty.
$$
Thus the claim \eqref{Claim2} is proved.

Now we may take the limit as $n\rightarrow\infty$ in \eqref{B} and obtain that
\begin{equation}\label{B2}
\begin{aligned}
\int_{\mathbb{R}^{3}}\nabla\psi_0\nabla\varphi dx=5\int_{\mathbb{R}^{3}}U_{0}^{4}\psi_{0}\varphi dx,
\end{aligned}
\end{equation}
 for any $\varphi\in D^{1,2}(\mathbb{R}^{3})$. Hence $\psi_{0}$ satisfies
\begin{equation}\label{C}
-\Delta\psi_{0}-5U_{0}^{4}\psi_{0}=0.
\end{equation}
By the nondegeneracy of $U_{0}$, we know
\begin{equation}\label{NR}
\psi_{0}\in\mathbf{span}\{\partial_{1}U_{0},\partial_{2}U_{0},\partial_{3}U_{0},\partial_{t}U_{0}\}.
\end{equation}

We prove that $\psi_{0}\neq0$. We take $\varphi_{n}=\psi_{n}$ in equation \eqref{B} to get
\begin{equation}\label{NTrival}
\begin{aligned}
\int_{\mathbb{R}^{3}}|\nabla\psi_n|^2 dx=2^{\ast}_{\mu_n}\int_{\mathbb{R}^{3}}\big(I_{\mu_{n}}\ast(U^{2^{\ast}_{\mu_{n}}-1}_{n}\psi_n)\big)U^{2^{\ast}_{\mu_{n}}-1}_{n}\psi_n dx+(2^{\ast}_{\mu_n}-1)\int_{\mathbb{R}^{3}}\big(I_{\mu_{n}}\ast(U^{2^{\ast}_{\mu_{n}}}_n)\big)U^{2^{\ast}_{\mu_{n}}-2}_{n}\psi^2_n dx.
\end{aligned}
\end{equation}
However, on one hand,
$$
\int_{\mathbb{R}^{3}}|\nabla\psi_n|^2 dx=1.
$$
On the other hand, we can repeat the arguments in Claims \eqref{Claim1} and \eqref{Claim2} to get
$$
2^{\ast}_{\mu_n}\int_{\mathbb{R}^{3}}\big(I_{\mu_{n}}\ast(U^{2^{\ast}_{\mu_{n}}-1}_{n}\psi_n)\big)U^{2^{\ast}_{\mu_{n}}-1}_{n}\psi_n dx+(2^{\ast}_{\mu_n}-1)\int_{\mathbb{R}^{3}}\big(I_{\mu_{n}}\ast(U^{2^{\ast}_{\mu_{n}}}_n)\big)U^{2^{\ast}_{\mu_{n}}-2}_{n}\psi^2_n dx\rightarrow5\int_{\mathbb{R}^{3}}U_{0}^{4}\psi^2_{0}dx,
$$
that is
$$
\int_{\mathbb{R}^{3}}U_{0}^{4}\psi^{2}_{0}dx=\frac{1}{5}.
$$
Therefore $\psi_{0}\neq0$.

Since we know
$$
\psi_{n}\in T_{z}{Z_{\mu}}^\bot=\mathbf{span}\{\partial_{1}U_0,\partial_{2}U_0,\partial_{3}U_0,\partial_{t}U_0\}^\bot,
$$
then for every
$$
\eta_0=aD_{t}U_{0}+\mathbf{b}\cdot\nabla U_{0}\in \mathbf{span}\{\partial_{1}U_0,\partial_{2}U_0,\partial_{3}U_0,\partial_{t}U_0\},
$$
where $a\in\mathbb{R},\mathbf{b}=(b_{1},b_{2},b_{3})\in\mathbb{R}^{3}$, we have
$$
<\psi_{n},\eta_{0}>=0
$$
Where we denote $<\cdot,\cdot>$ as the inner product in $D^{1,2}(\mathbb{R}^{3})$. However, as $n\rightarrow\infty$, we know
$$
<\psi_{0},\eta_{0}>=0.
$$
This contradicts to  \eqref{NR} that $\psi_{0}\in\mathbf{span}\{\partial_{1}U_{0},\partial_{2}U_{0},\partial_{3}U_{0},\partial_{t}U_{0}\}$, since we had proved that $\psi_{0}\neq0$. Hence any solution satisfies \eqref{A} must belong to $T_{z}Z$ in the space $L^{2^{\ast}}$, that is $T_{z}{Z_{\mu}}=\mathbf{Ker}{[I_{\mu}^{\prime\prime}(z)]}$, we finish the proof.

Finally, it is easy to check that operator $I_{\mu}^{\prime\prime}(z)$ is Fredholm and
$$
dim(\mathbf{ker}[I_{\mu}^{\prime\prime}(z)])-codim(\mathbf{Im}[I_{\mu}^{\prime\prime}(z)])=0.
$$
Thus for all $z\in Z_{\mu}$, $I_{\mu}^{\prime\prime}(z)$ is an index 0 Fredholm map. Combining with Theorem \ref{TH3}, we conclude that the critical manifold $Z_{\mu}$ is nondegenerate when $\mu$ close to $N=3$. $\hfill{} \Box$

\end{document}